\documentclass[13pt]{article}  
\usepackage{amssymb}              
\usepackage{amsthm}
\usepackage{amsmath}
\usepackage{eucal}
\usepackage{verbatim}
\usepackage[dvips]{graphicx}
\usepackage{multirow}
\usepackage{fancyhdr}
\usepackage{color}
\usepackage{enumerate}
\usepackage{calrsfs}
\usepackage{fullpage}
\usepackage{amsmath}
\usepackage{eufrak}

\usepackage[pagewise]{lineno}

\newtheorem{theorem}{Theorem}
\newtheorem{lemma}{Lemma}

\newtheorem{remark}{Remark}

\makeatletter
\newcommand{\leqnomode}{\tagsleft@true}
\newcommand{\reqnomode}{\tagsleft@false}
\makeatother

\def\({\begin{eqnarray}}
\def\){\end{eqnarray}}
\def\[{\begin{eqnarray*}}
\def\]{\end{eqnarray*}}
\def\part#1#2{\frac{\partial #1}{\partial #2}}

\def\R{\mathbb{R}}
\def\N{\mathbb{N}}
\def\d{\mathrm{d}}
\def\tot#1#2{\frac{\d #1}{\d #2}}
\def\eps{\varepsilon}

\def\wtx{\widetilde x}

\def\bPsi{{\widetilde\Psi}}
\def\bPsii{\bPsi^i}

\def\bPhi{{\widetilde\Phi}}
\def\bPhii{\bPhi^i}

\def\upsi{{\underline\psi}}

\begin{document}

\title{Asymptotic behavior of the linear consensus model with delay and anticipation}   
\author{Jan Haskovec\footnote{Computer, Electrical and Mathematical Sciences \& Engineering, King Abdullah University of Science and Technology, 23955 Thuwal, KSA.
jan.haskovec@kaust.edu.sa}}

\date{}

\maketitle

\begin{abstract}
We study asymptotic behavior of solutions of the first-order linear consensus model
with delay and anticipation, which is a system of neutral delay differential equations.
We consider both the transmission-type and reaction-type delay that
are motivated by modeling inputs.
Studying the simplified case of two agents, we show that,
depending on the parameter regime,
anticipation may have both a stabilizing and destabilizing effect on the solutions.
In particular, we demonstrate numerically that moderate level of anticipation generically
promotes convergence towards consensus, while too high level disturbs it.
Motivated by this observation, we derive sufficient conditions
for asymptotic consensus in the multiple-agent systems,
which are explicit in the parameter values delay length and anticipation level,
and independent of the number of agents.
The proofs are based on construction of suitable Lyapunov-type functionals.
\end{abstract}
\vspace{2mm}

\vspace{2mm}


\textbf{2010 MR Subject Classification}: 34K05, 82C22, 34D05, 92D50.
\vspace{2mm}

\section{Introduction}\label{sec:Intro}

In this paper we study asymptotic behavior of the linear consensus model
consisting of a group of $N\in\N$ agents, each of them having their own opinion
or assessment of a certain quantity, represented by the real vector
$x_i\in\R^d$, $i\in [N]$, with $d\geq 1$ the space dimension.
Here and in the sequel we denote the set $[N] := \{1, 2, \ldots,N\}$.
Each of the agents communicates with all others and
revises its own opinion based on a weighted average of
the perceived other agents' opinions.
If this is done continuously in time,
we have $x_i=x_i(t)$ for $t\geq 0$ and
\(  \label{eq:0}
   \dot x_i(t) = \sum_{j=1}^N \psi_{ij}(t) (\xi_j(t) - \xi_i(t)).
\)
Here $\xi_j=\xi_j(t)$ represents the information
that at time $t>0$ is available about agent $j$'s opinion
to all other agents.
The communication weights $\psi_{ij}=\psi_{ij}(t)$ measure the intensity
of the influence between agent $i$ and agent $j$.
The choice $\xi_j(t) := x_j(t)$ for all $j\in [N]$
and $\psi_{ij}(t) := \psi(|x_i(t) - x_j(t)|)$
gives the very well studied Hegselmann-Krause system \cite{HK},
which is generic for modeling self-organized consensus
in systems of interacting agents
\cite{Camazine, Castellano, Jadbabaie, Krugman, Naldi, Vicsek, Xu}.

For many applications in biological and socio-economical systems or control problems (for instance, swarm robotics \cite{Hamman, E3B, Valentini}),
it is natural to include a time delay (lag) in the model reflecting the time needed for each agent to receive information from other agents,
and/or to react to it. In this paper we also assume that the agents are able to anticipate the action of their conspecifics
by measuring and extrapolating the momentary rate of change of their state vectors.
However, this information is as well subject to the time lag.
In general, the length of the delay $\tau$ may depend both on the state of the system
and on the external conditions, i.e., it may be different for each pair $(i,j)$, it may vary in time
or be random, following a certain distribution.
For simplicity, in this paper we consider a fixed, globally constant delay $\tau>0$.
From the modeling point of view, it is reasonable to consider the following two types of systems.

\begin{itemize}
\item
\textbf{System without self-delay (transmission-type delay),}
which reflects the situation when each agent receives information from its surroundings
with a certain time lag due to finite speed of information transmission and/or perception.
Therefore, agent $i$ with opinion $x_i(t)$ receives at time $t>0$ the information about the opinion of agent $j$
in the form $x_j(t-\tau)$. Also, it is able to measure the rate of change of agent $j$'s opinion,
which it receives as $\dot x_j(t-\tau)$, and extrapolates.
Consequently, we have
\[
   \xi_j(t) := x_j(t-\tau) + \lambda\tau \dot x_j(t-\tau),\qquad   \xi_i(t) := x_i(t),
\]
and \eqref{eq:0} turns into
\( \label{eq:prop}
   \dot x_i(t) = 
     \sum_{j\neq i} \psi_{ij}(t) \bigl( x_j(t-\tau) + \lambda\tau \dot x_j(t-\tau) - x_i(t)\bigr), \qquad i\in [N].
\)
The parameter $\lambda\geq 0$ denotes the strength of anticipation. The value $\lambda=1$
corresponds to the Taylor expansion of the smooth curve $x_j=x_j(t)$ at time $t-\tau$
to first order; $\lambda>1$ corresponds to (over)extrapolation, while $\lambda<1$
can be seen as interpolation between $x_j(t-\tau)$ and $x_j(t-\tau) + \tau\dot x_j(t-\tau)$.
Note that the summation in \eqref{eq:prop} runs through all $j\in [N]$ such that $j\neq i$.
This prevents $x_i$ from interacting with its own opinion from the past,
which obviously is not justified from the modeling point of view.

\item
\textbf{System with self-delay (reaction-type delay),}
which takes into account the time needed for the agents to react to the information that they already received.
In this case, the agent $i$'s reaction executed at time $t$
is based upon the state of the system, including agent $i$'s own state, and its rate of change, at time $t-\tau$.
We thus have
\[
   \xi_j(t) := x_j(t-\tau) + \lambda\tau \dot x_j(t-\tau),\qquad   \xi_i(t) := x_i(t-\tau) + \lambda\tau \dot x_i(t-\tau),
\]
and \eqref{eq:0} takes the form
\( \label{eq:react}
   \dot x_i(t) = 
     \sum_{j=1}^N \psi_{ij}(t) \bigl( x_j(t-\tau) + \lambda\tau \dot x_j(t-\tau) - x_i(t-\tau) - \lambda\tau \dot x_i(t-\tau)\bigr), \qquad i\in [N].
\)
\end{itemize}
In both cases the system is equipped with the initial datum
\(\label{IC}
   x_i(t) = x^0_i(t),\qquad i\in [N], \quad t \in [-\tau,0],
\)
with prescribed continuously differentiable trajectories $x^0_i\in{C}^1([-\tau,0])$, $ i\in [N]$.
We note that both systems \eqref{eq:prop} and \eqref{eq:react}
fall into the class of neutral delay differential equations, see, e.g. \cite[Chapter 9]{Hale-book}.

As already noted, in the classical Hegselmann-Krause bounded confidence
model the communication weights $\psi_{ij}=\psi_{ij}(t)$
depend nonlinearly on the dissimilarity $|x_i(t) - x_j(t)|$ of the opinions
of the respective agents.
In this paper we resort to prescribing constant communication weights
$\psi_{ij}$ for all $i,j\in [N]$.
This restriction, not unusual in previous literature \cite{E3A, E4, E6},
is motivated by the analytical approach that we shall employ.
It is based on constructing Lyapunov functionals involving,
for the system \eqref{eq:prop}, terms of the type
\[
   x_i - \lambda\tau\sum_{j\neq i} \psi_{ij} x_j(t-\tau).
\]
The time derivative of this term along the solutions of \eqref{eq:prop} reads
\[
   \tot{}{t} \left( x_i - \lambda\tau\sum_{j\neq i} \psi_{ij} x_j(t-\tau) \right)
      = \sum_{j\neq i} \psi_{ij} \bigl( x_j(t-\tau) - x_i(t)\bigr),
\]
and the key observation is that it consists of derivative-free terms only,
which can be suitably estimated.
Similarly, for \eqref{eq:react} we shall use
\[
   \tot{}{t} \left( x_i - \lambda\tau\sum_{j\neq i} \psi_{ij} \bigl(x_j(t-\tau)-x_i(t-\tau) \bigr) \right)
      = \sum_{j\neq i} \psi_{ij} \bigl( x_j(t-\tau) - x_i(t-\tau)\bigr).
\]
However, the systems \eqref{eq:prop} and \eqref{eq:react} with constant
communication weights can be seen as linearizations of the respective
nonlinear systems about their steady states.
The sufficient conditions for the global asymptotic consensus
in the linear systems that we shall present in Section \ref{sec:main}
below can then be seen as conditions for local asymptotic stability 
of the respective nonlinear steady states.
Further structural assumptions on the properties of the matrix $(\psi_{ij})_{i,j\in [N]}$
shall be specified in Section \ref{sec:main}.

The main question in the context of consensus formation models
of the type \eqref{eq:0} is whether the dynamics of continuously evolving opinions
tends to an (asymptotic) emergence of one or more opinion clusters \cite{JM} formed
by agents with (almost) identical opinions.
In case of global communications (i.e., $\psi_{i,j}>0$ for all ${i,j\in [N]}$),
one is interested in \emph{global consensus}, which is the state where
all agents have the same opinion $x_i=x_j$ for all $i,j \in [N]$.
For the Hegselmann-Krause model with delay and its various modifications,
this question has been studied extensively in numerous works,
see, e.g., \cite{CPP, Has-ProcAMS, Has-BLMS, Has-SIADS, Liu-Wu, Lu, E2, Paolucci, Pignotti-Trelat, E3A, E4, E6}.
Anticipation in the context of collective dynamics was introduced in \cite{Shu-Tadmor},
however, without delay or time lag.
The authors studied large-time behavior of systems driven by radial potentials, which
react to anticipated positions. 
As a special case, they considered the Cucker-Smale model
of alignment \cite{CS1, CS2}, proving the decisive role of anticipation
in driving such systems with attractive potentials into velocity alignment
and spatial concentration. They also studied the concentration effect near equilibrium
for anticipation-based dynamics of pair of agents governed by attractive--repulsive
potentials.

However, to our best knowledge,
none of the previous works considered the first-order consensus model
of the type \eqref{eq:0} with delay and anticipation,
i.e., no results are found in the literature concerning the asymptotic behavior
of the systems \eqref{eq:prop} and \eqref{eq:react} or their variants.
The main goal of this paper is to fill this gap.
It is known that the presence of delay generically causes destabilization of solutions
through appearance of damped or undamped oscillations \cite{Hale-book, Smith, Gyori-Ladas}.
One may then conjecture that anticipation could oppose this effect
by re-establishing stability in certain parameter regimes,
which in the context of the systems \eqref{eq:prop}, \eqref{eq:react}
means promotion of convergence of solutions towards global consensus.
To gain insights into this hypothesis, we first study the simplified case
of two-agent systems, where \eqref{eq:prop} and \eqref{eq:react} reduce
reduce to single neutral delay differential equations.
Combining known analytical results and numerical simulations,
we shall demonstrate that moderate levels of anticipation indeed
promote stability. In particular, for the system \eqref{eq:prop} with transmission-type delay,
moderate anticipation reduces the amplitude of the oscillations of the solutions.
For the system \eqref{eq:react} with reaction-type delay 
we shall numerically identify a parameter range where anticipation 
promotes stability in otherwise unstable solutions through a Hopf-type bifurcation.
Noting that anticipation may well have a destabilizing effect on the dynamics
motivates the search for sufficient conditions that guarantee
convergence towards global consensus in the systems \eqref{eq:prop} and \eqref{eq:react}.
We shall provide such conditions explicitly in the parameter values
$\lambda\geq 0$ and $\tau\geq 0$ in the respective systems with constant communication weights.
These results, formulated as Theorems \ref{thm:prop} and \ref{thm:react} below,
are, to our best knowledge, new.
The proofs will be based on construction of suitable Lyapunov functionals
and, despite the linearity of the considered systems,
highly nontrivial due to the presence of anticipation and delay terms.

This paper is organized as follows. In Section \ref{sec:main}
we present our main results on the sufficient conditions for global asymptotic consensus
for the systems \eqref{eq:prop} and \eqref{eq:react},
discussing in detail the assumptions on the communication weights that we adopt.
In Section \ref{sec:N2} we consider the simplified setting with two agents only, $N=2$,
where \eqref{eq:prop} and \eqref{eq:react} reduce to single neutral delay differential equations.
We then provide insights into the dynamics of both systems by discussing known analytical results
and presenting results of numerical simulations.
In Section \ref{sec:prop} we provide the proof of our main result on global asymptotic consensus
for the transmission-type delay system \eqref{eq:prop}, and
in Section \ref{sec:react} we provide the proof for the reaction-type delay system \eqref{eq:prop}.

\section{Assumptions and main results}\label{sec:main}
In this section we introduce the particular assumptions
on the communication weights $\psi_{ij}$ appearing in the systems
\eqref{eq:prop} and \eqref{eq:react}.
As explained in the Introduction, we shall work with fixed, time-independent
weights in this paper, which facilitates analysis of the asymptotic behavior of the solutions
using Lyapunov-type functionals. As the the mathematical properties of the systems
\eqref{eq:prop} and \eqref{eq:react} differ, we shall introduce two different sets
of assumptions on $\psi_{ij}$ below. Then, we shall formulate two theorems providing
sufficient conditions for the asymptotic stability of the trivial steady state,
i.e., sufficient conditions for reaching asymptotic consensus.

Let us remark that both \eqref{eq:prop} and \eqref{eq:react}
are instances of linear neutral functional differential equations.
Consequently, by standard results, see, e.g., \cite[Chapter 9]{Hale-book},
they possess unique global solutions
subject to the continuously differentiable initial datum \eqref{IC}.

For later use let us define the mean opinion
\(  \label{def:mean}
    X(t) := \frac{1}{N} \sum_{i=1}^N x_i(t)
\)
and the maximum opinion discord
\(  \label{def:dx}
   d_x(t) := \max_{i,j \in [N]} |x_i(t) - x_j(t)|.
\)
Clearly, global consensus corresponds to $d_x=0$.

\subsection{Transmission-type delay} 
For the system \eqref{eq:prop} without self-delay, we assume that
the matrix $(\psi_{ij})_{i,j\in [N]}$ of communication weights is row-stochastic in the sense
\(  \label{psi:conv}
   \sum_{j\neq i} \psi_{ij} = 1 \qquad\mbox{for all } i\in [N].
\)
As the diagonal elements $\psi_{ii}$ do not appear in \eqref{eq:prop},
we may set them to zero and treat the whole matrix $(\psi_{ij})_{i,j\in [N]}$
as row stochastic. Moreover, we assume that all non-diagonal elements
are strictly positive,
\(   \label{psi:pos}
   \psi_{ij} > 0 \qquad\mbox{for all } i\neq j.
\)
Let us observe that the system \eqref{eq:prop}, in general,
does not conserve the mean \eqref{def:mean},
and even an eventual symmetry of the matrix $(\psi_{ij})_{i,j\in [N]}$,
which we however do \emph{not} assume in the sequel,
would not guarantee conservation of the mean.
Consequently, the global consensus value, if reached, cannot be
easily inferred from the initial datum and can be seen
as an emergent property of the system.
As we shall see in Section \ref{sec:prop},
this also brings some additional difficulties
for the analysis of the asymptotic stability of the global consensus state.

Furthermore, we assume that there exists some $T_0>0$ and fixed indices $I,J \in [N]$ such that
\(  \label{IJ}
   d_x(t) = |x_{I}(t) - x_{J}(t)| \quad\mbox{for } t\geq T_0.
\)
Admittedly, the validity of this assumption is difficult to examine a priori,
although our extensive numerical investigations suggest its generic plausibility.
The necessity for adopting this assumption stems from the form
of the Lyapunov functional that we shall construct in order to study
the asymptotic behavior of system \eqref{eq:prop} and will be explained
in detail in Remark \ref{rem:IJ} of Section \ref{sec:prop}.

\begin{theorem}\label{thm:prop}
Let $N\geq 3$ and let the matrix $(\psi_{ij})_{i,j\in [N]}$ of communication weights satisfy
the row stochasticity \eqref{psi:conv} and positivity \eqref{psi:pos}.
Let the assumption \eqref{IJ} be satisfied and let
\(   \label{cond:thm:prop}
   \lambda\tau \leq 1.
\)
Then all solutions of \eqref{eq:prop} subject to the initial datum \eqref{IC}
converge to a consensus state, i.e.,
$\lim_{t\to\infty} d_x(t) = 0$.
\end{theorem}

Theorem \ref{thm:prop} gives sufficient conditions for reaching global consensus
as $t\to\infty$ in the system \eqref{eq:prop} with at least three agents.
As will be discussed in Section \ref{sec:N2}, for $N=2$ asymptotic consensus
is reached if and only if $|\lambda\tau|<1$.

It is well known \cite[Theorem 2.1]{Has-SIADS} that without anticipation ($\lambda=0$),
system \eqref{eq:prop} with row-stochastic interaction matrix
reaches asymptotic consensus unconditionally,
i.e., for all initial data and for all values of the delay $\tau\geq 0$.
Consequently, anticipation may only provide the benefit
of reducing the amplitude of eventual oscillations of the solutions.
This is indeed the case for small enough values of $\lambda>0$,
as we shall demonstrate numerically in Section \ref{subsec:N2prop}.
However, choosing $\lambda$ too large leads to loss of stability.
The sufficient condition \eqref{cond:thm:prop} of Theorem \ref{thm:prop}
can therefore be understood as an explicit quantitative guarantee in terms of the parameter
values $\lambda$ and $\tau$ that stability is preserved.

\subsection{Reaction-type delay} 
The system \eqref{eq:react} with self-delay has the remarkable
property that, if the matrix of communication weights $(\psi_{ij})_{i,j\in [N]}$
is symmetric,
then \eqref{eq:react} does conserve the mean \eqref{def:mean}.
Then, if a global consensus is reached, it is determined by the initial condition
through $X(0)$.

Moreover, we assume that the matrix $(\psi_{ij})_{i,j\in [N]}$ is irreducible.
Note that a matrix is irreducible if and only if it represents
a (weighted) connectivity matrix of a strongly connected graph
(directed graph is called strongly connected if there is a path in each
direction between each pair of its vertices).
Clearly, this is a minimal necessary condition for reaching global consensus.

\begin{theorem}\label{thm:react}
Let the matrix $(\psi_{ij})_{i,j\in [N]}$ of communication weights be
symmetric
and irreducible.
Let
\(  \label{cond:react}
   (1+\lambda)\tau < 1/2.
\)
Then all solutions of \eqref{eq:react} subject to the initial datum \eqref{IC}
converge to the consensus state $X(0)$, i.e.,
\[
   \lim_{t\to\infty} x_i(t) = X(0) \qquad\mbox{for all } i\in [N],
\]
with the mean $X(0)$ defined in \eqref{def:mean}.
\end{theorem}

Condition \eqref{cond:react} of Theorem \ref{thm:react}
can be read as a smallness condition on $\lambda>0$ relative to a given
$\tau < 1/2$. As we shall observe numerically in Section \ref{subsec:N2react} for $N=2$,
there is a parameter regime where solutions with moderate anticipation, i.e., with $\lambda>0$,
are stable (converge to consensus), while solutions without anticipation, i.e., $\lambda=0$,
are unstable. Obviously, Theorem \ref{thm:react} does not capture this phenomenon.
The reason is that its proof is based on a construction of a Lyapunov functional,
which naturally leads to a sufficient condition 
that imposes smallness of $\lambda$ and $\tau$.
In fact, we are not aware of any analytical results that would
explain this stabilization-by-anticipation phenomenon.
On the other hand, too large values of $\lambda>0$ have the opposite effect,
i.e., they cause loss of stability of solutions that are stable without anticipation.
Therefore, the value of the statement of Theorem \ref{thm:react}
lies in guaranteeing preservation of stability.

\section{The simplified case - two agents}\label{sec:N2}
To gain insight into the type of dynamics produced by the systems \eqref{eq:prop} and \eqref{eq:react},
we consider their maximally simplified versions with two agents only, $N=2$.
Note that then \eqref{psi:conv} dictates $\psi_{12}=\psi_{21}=1$ for the transmission-type delay system \eqref{eq:prop},
and we shall use the same values also for the reaction-type delay system \eqref{eq:react}.

\subsection{Transmission delay with $N=2$}\label{subsec:N2prop}
The system \eqref{eq:prop} with $N=2$ and $\psi_{12}=\psi_{21}=1$ reads
\[
   \dot x_1 &=& \wtx_2 + \lambda\tau \dot\wtx_2 - x_1, \\
   \dot x_2 &=& \wtx_1 + \lambda\tau \dot\wtx_1 - x_2.
\]
Denoting $x:=x_1-x_2$, we have
\(    \label{eq:N2trans}
   \dot x = -\wtx - \lambda\tau \dot\wtx - x,
\)
equipped with the initial datum $x(t) = x^0(t)$ for $t\in [-\tau,0]$,
with $x^0 \in C^1([-\tau,0])$.
Clearly, consensus in the context of \eqref{eq:N2trans} corresponds to $x=0$.
It is known that the zero steady state of \eqref{eq:N2trans} is globally stable 
if and only if $|\lambda\tau|<1$, see \cite[Chapter 9.9]{Hale-book}.
In other words, for a given $\tau> 0$, anticipation
in \eqref{eq:N2trans} leads to consensus, for all initial data $x^0 \in C^1([-\tau,0])$,
if and only if $\lambda < \tau^{-1}$.
Note that without anticipation (i.e., $\lambda=0$), the solution
always converges to the consensus $x=0$, for any length of the delay $\tau\geq 0$.

We illustrate this fact by a numerical simulation. We discretize \eqref{eq:N2trans}
on an equidistant mesh with time step $\Delta t >0$, chosen such that
$\tau$ is an integer multiple of $\Delta t$. 
Implicit Euler discretization of $\dot x$ and $\dot\wtx$ leads then to a difference
equation that can be numerically integrated.
We show results of numerical integration with the constant initial datum $x^0\equiv 1$.
In Fig. \ref{fig:1} we consider the case $\tau:=0.25$, where the solution of \eqref{eq:N2trans} without anticipation (i.e., $\lambda:=0$)
decreases monotonically without oscillations. With $\lambda:=1$, which corresponds to Taylor expansion
of the respective term, the solutions still decays monotonically, but the decay is slower, as is evident from the inserted logarithmic plot.
Choosing $\lambda:=4$, which gives the borderline case $\lambda\tau=1$, oscillations appear
that convergence to a periodic solution.
Finally, after crossing the threshold $\lambda\tau=1$ by choosing $\lambda:=4.5$,
diverging oscillations appear.
In Fig. \ref{fig:2} we choose $\tau:=1.25$, where the solution of \eqref{eq:N2trans} without anticipation
develops oscillations (but still converges to zero as $t\to\infty$).
We first choose $\lambda:=0.2$ to demonstrate that small enough anticipation
may actually decrease the amplitude of the oscillations.
However, already the choice $\lambda:=0.6$ leads to an increase in the amplitude
compared with the no-anticipation case. Finally, the Taylor anticipation $\lambda:=1$
crosses the boundary of stability $\lambda\tau=1$ and produces a diverging solution.

\begin{figure} \centering
 \includegraphics[width=.32\textwidth]{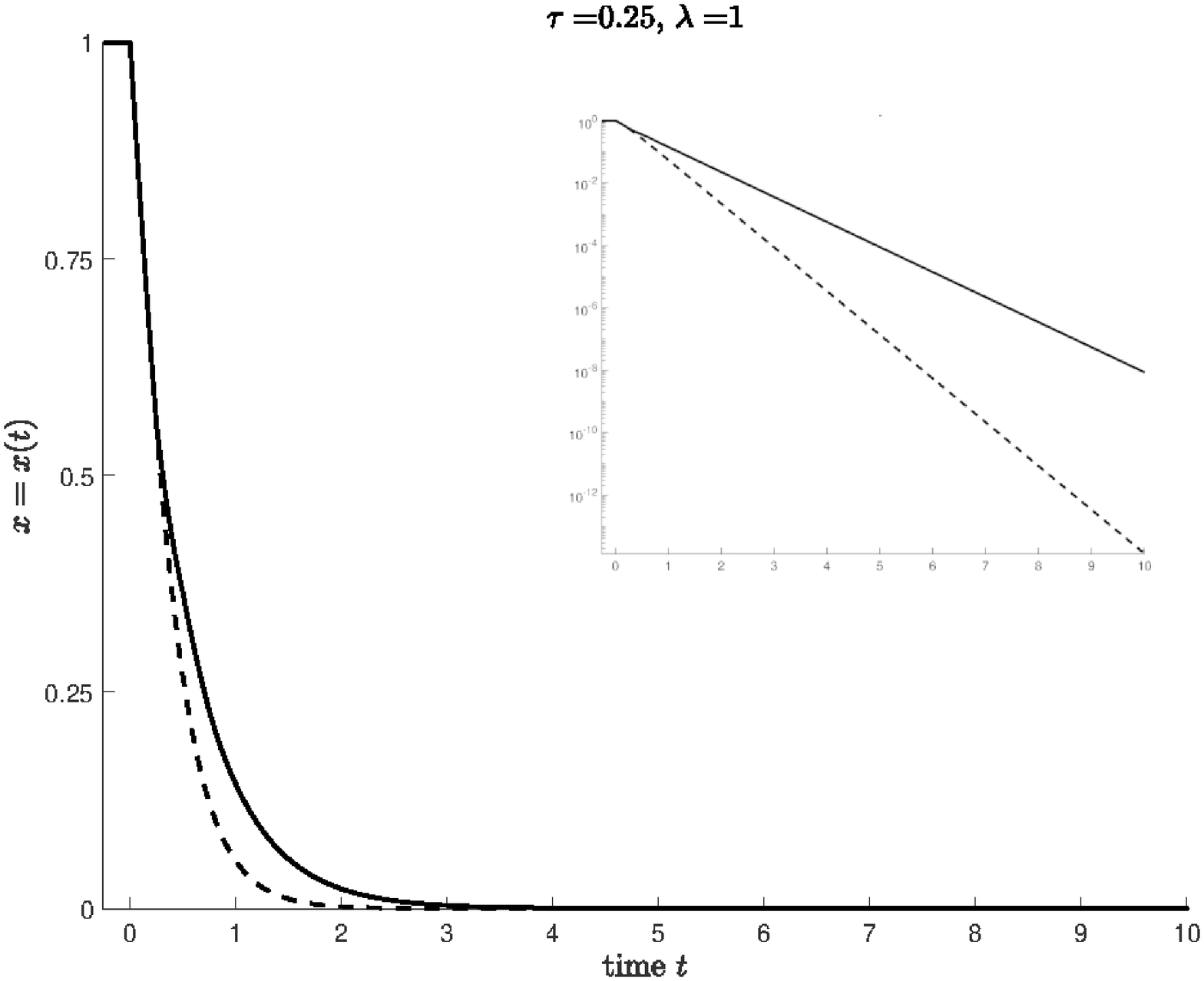}
 \includegraphics[width=.32\textwidth]{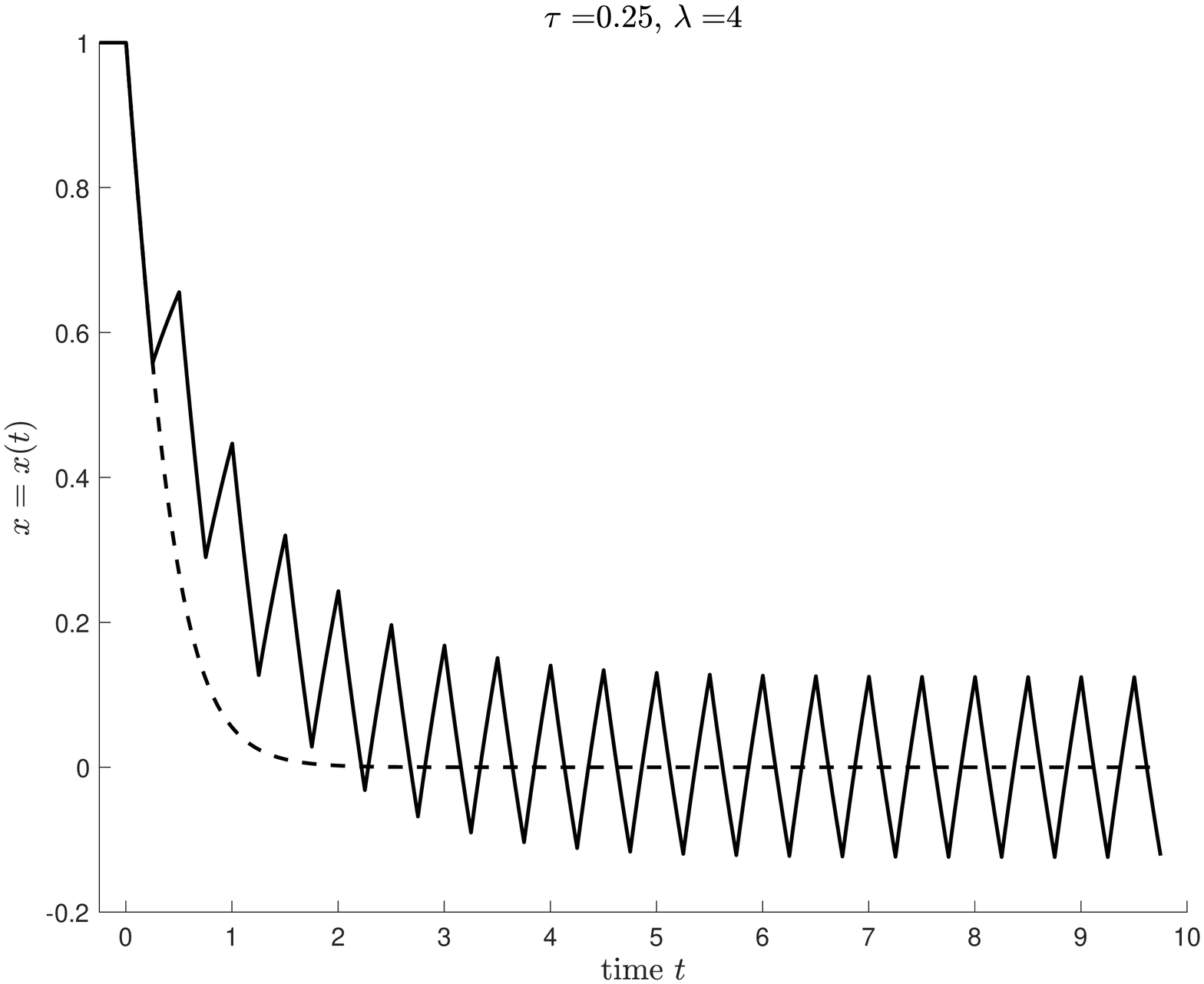}
 \includegraphics[width=.32\textwidth]{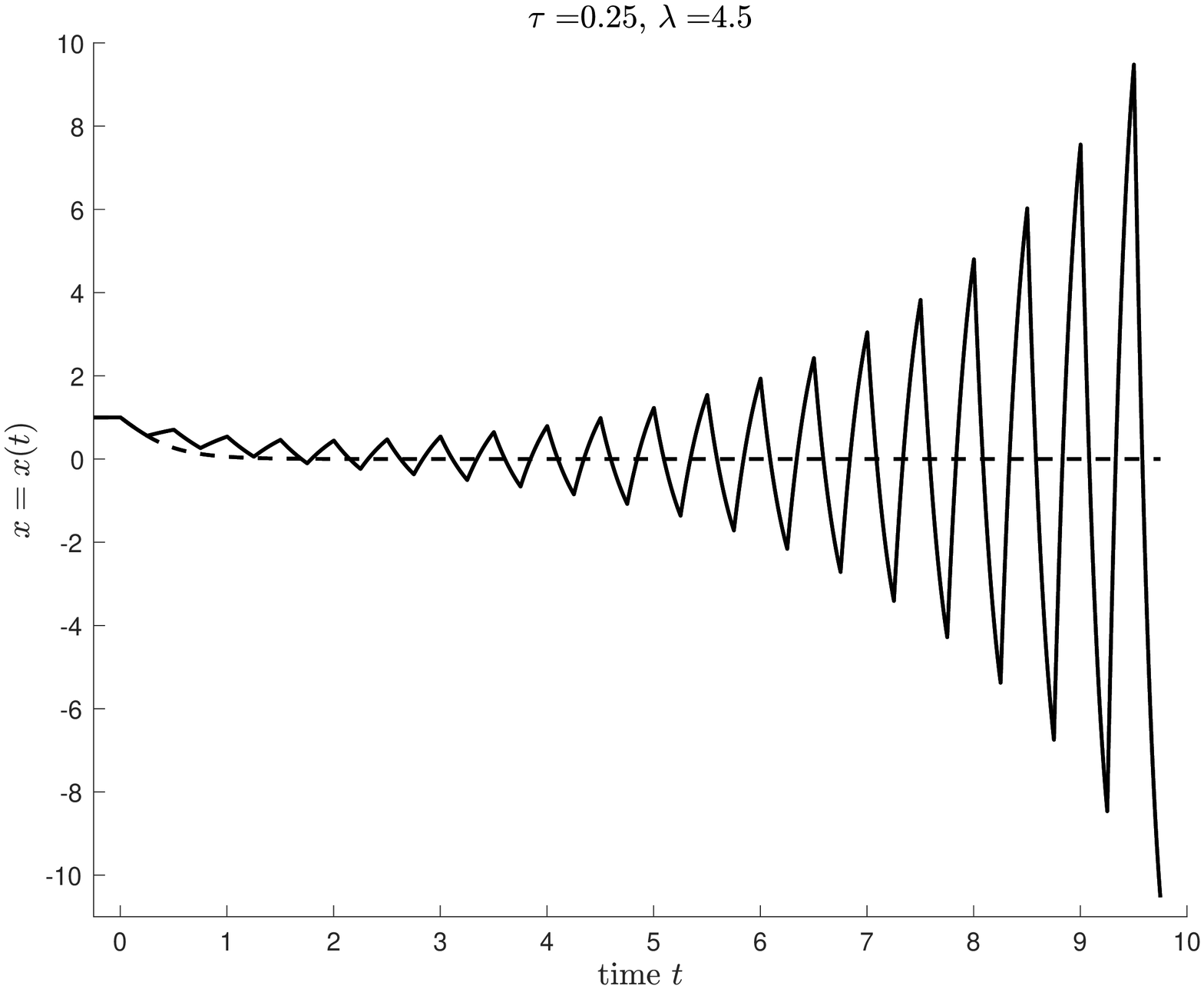}
 \caption{Numerical integration of \eqref{eq:N2trans} subject to the constant initial datum $x^0 \equiv 1$,
 with $\tau=0.25$ and $\lambda=1$ (solid line in left panel), $\lambda=4$ (solid line in middle panel)
 and $\lambda=4.5$ (solied line in right panel).
 For comparison, numerical solution of \eqref{eq:N2trans} with $\tau=0.25$ and $\lambda=0$ (no anticipation)
 is indicated by dashed line in all three panels.
 The inserted plot in the left panel is in logarithmic scale (vertical axis).
 \label{fig:1}}
\end{figure}

\begin{figure} \centering
 \includegraphics[width=.32\textwidth]{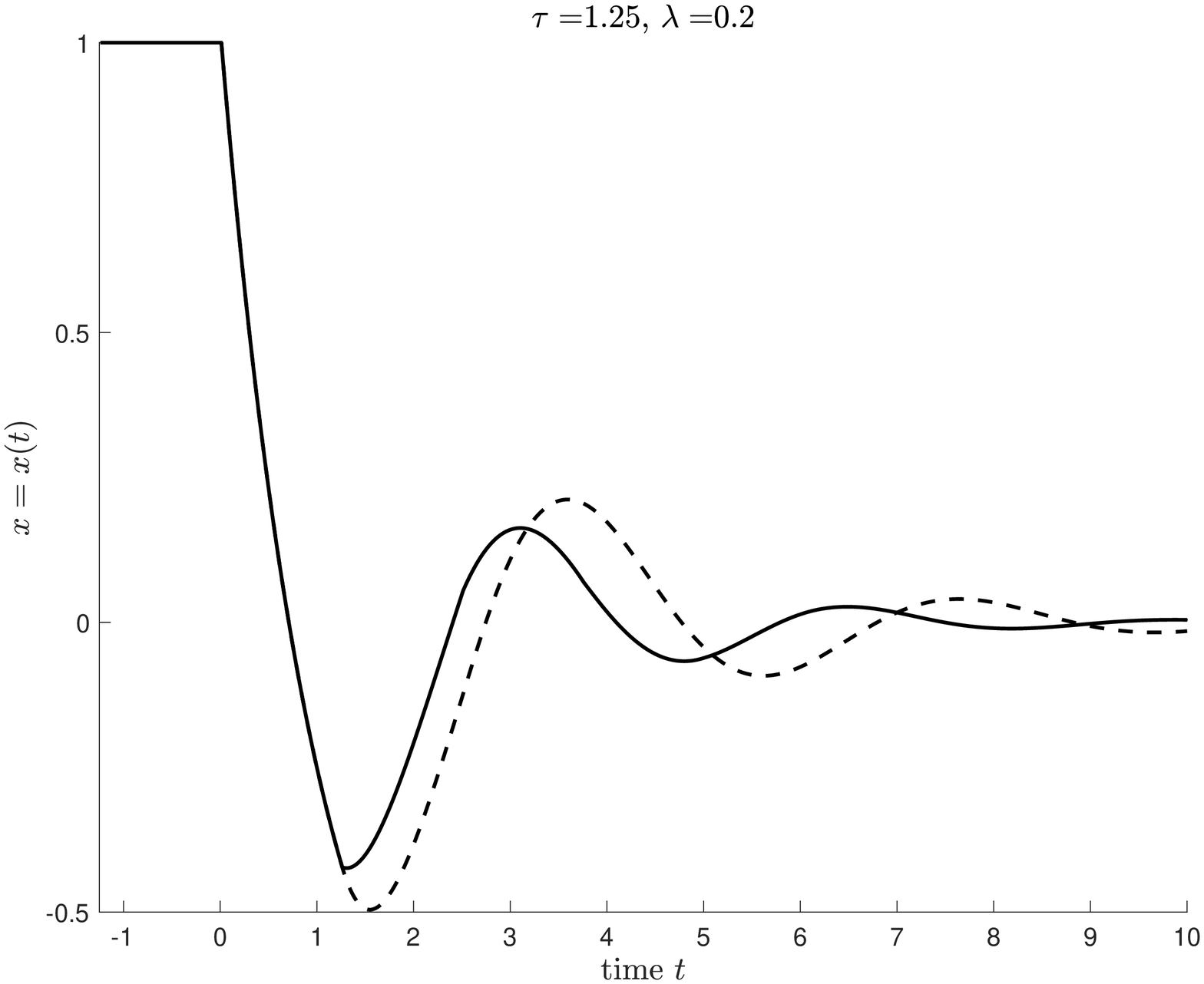}
 \includegraphics[width=.32\textwidth]{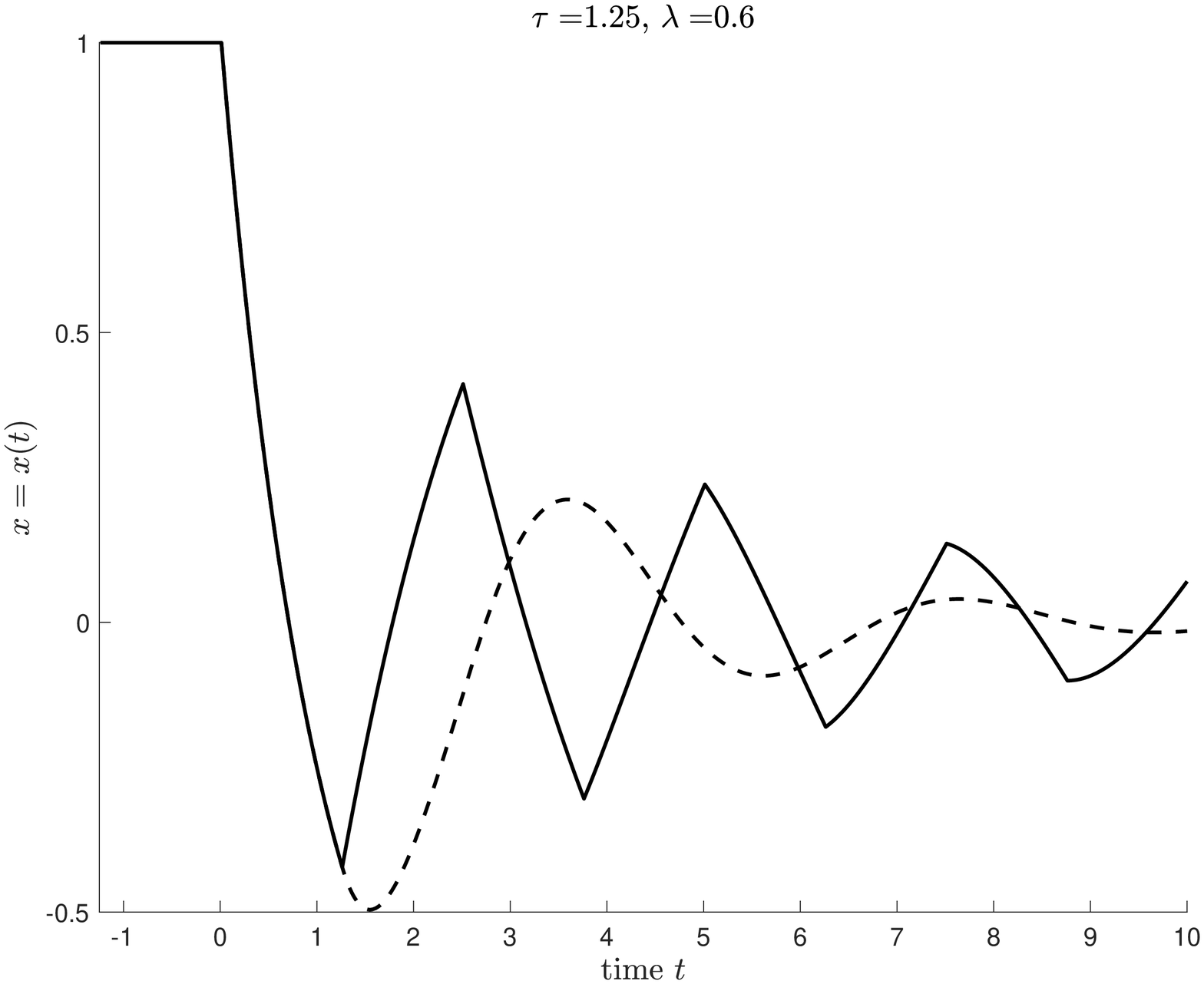}
 \includegraphics[width=.32\textwidth]{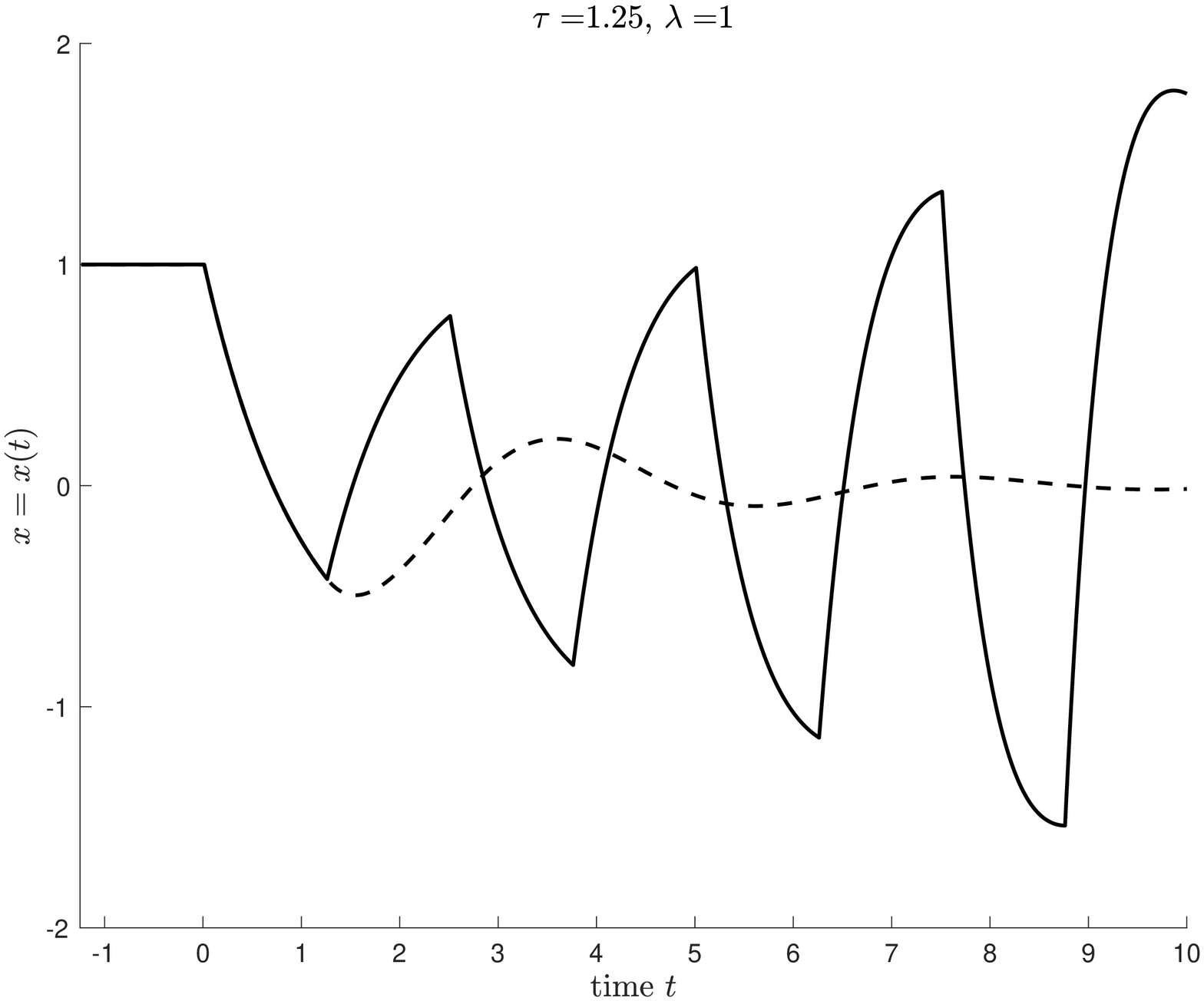}
 \caption{Numerical integration of \eqref{eq:N2trans} subject to the constant initial datum $x^0 \equiv 1$,
 with $\tau=1.25$ and $\lambda=0.2$ (solid line in left panel), $\lambda=0.6$ (solid line in middle panel)
 and $\lambda=1$ (solied line in right panel).
 For comparison, numerical solution of \eqref{eq:N2trans} with $\tau=1.25$ and $\lambda=0$ (no anticipation)
 is indicated by dashed line in all three panels.
 \label{fig:2}}
\end{figure}

\subsection{Reaction delay with $N=2$}\label{subsec:N2react}
The system \eqref{eq:react} with $N=2$ and $\psi_{12}=\psi_{21}=1$ reads
\[
   \dot x_1 &=& \wtx_2 + \lambda\tau \dot\wtx_2 - (\wtx_1 + \lambda\tau\dot \wtx_1), \\
   \dot x_2 &=& \wtx_1 + \lambda\tau \dot\wtx_1 - (\wtx_2 + \lambda\tau\dot \wtx_2).
\]
Denoting $x:=x_1-x_2$, we have
\(    \label{eq:N2react}
   \dot x = - 2 \wtx - 2\lambda\tau \dot\wtx,
\)
subject to the initial datum $x(t) = x^0(t)$ for $t\in [-\tau,0]$,
with $x^0 \in C^1([-\tau,0])$.
Although stability of equations of the type \eqref{eq:N2react} has been studied
in the literature, see, e.g., \cite{Kim}, no explicit analytical condition for stability
is known to us that would cover the case $\lambda\tau>0$ in \eqref{eq:N2react}.
If there is no anticipation, i.e., $\lambda:=0$, then \eqref{eq:N2react} reduces to
\(    \label{eq:N2noant}
   \dot x(t) = - 2 x(t-\tau),
\)
and an analysis of the corresponding characteristic equation
reveals that:
\begin{itemize}
\item
If $0 < 2\tau < e^{-1}$, then $x=0$ is asymptotically stable.
\item
If $e^{-1} < 2\tau < \pi/2$, then $x=0$ is asymptotically stable,
but every nontrivial solution of \eqref{eq:N2noant} is oscillatory
(i.e., changes sign infinitely many times as $t\to\infty$).
\item
If $2\tau > \pi/2$, then $x=0$ is unstable
and nontrivial solutions oscillate with unbounded amplitude
as $t\to\infty$.
\end{itemize}
We refer to Chapter 2 of \cite{Smith} and \cite{Gyori-Ladas} for details.

The analysis performed in Section \ref{sec:react} below, based on a construction of a Laypunov functional,
gives the sufficient condition
\(  \label{cond:suffN2}
   2(1+\lambda)\tau < 1
\)
for asymptotic stability of the trivial steady state. This is clearly not optimal, since we know that
for $\lambda=0$ the sufficient and necessary condition for asymptotic stability is $2\tau < \pi/2$.
Consequently, to get a more complete picture, we studied the asymptotic behavior of \eqref{eq:N2react} numerically.
Again, we discretized \eqref{eq:N2react} on an equidistant mesh with time step $\Delta t >0$, chosen such that
$\tau$ is an integer multiple of $\Delta t$. 
Implicit Euler discretization of $\dot x$ and $\dot\wtx$ leads then to a difference
equation that can be numerically integrated.
We chose the constant initial datum $x^0(t) \equiv 1$ for $t\in [-\tau,0]$
and performed systematic simulations with varying parameter values $\lambda\geq 0$ and $\tau\geq 0$,
detecting either decay or growth of the amplitude of the oscillations of the solution.
The result is presented in Fig. \ref{fig:3}, where the solid curve represents the threshold for stability
of the zero steady state obtained numerically -- the values of $(\lambda,\tau)$ below the solid curve
give $x(t)\to 0$ for large $t>0$, while the values above the curve produce diverging $x(t)$.
For comparison, the dashed curve represents the analytical sufficient condition \eqref{cond:suffN2}.
The critical value $\tau=\pi/4$ for the system without anticipation is indicated by the dotted line.

The most interesting set of parameter values in Fig. \ref{fig:3} is located in the area enclosed between
the solid curve and the dotted line $\tau=\pi/4$. Here the zero steady state is stable
with anticipation, but becomes unstable when anticipation is suppressed (i.e., $\lambda=0$).
This stabilizing effect of anticipation only works for values of $\tau$ below approx. $0.9$, with the range of admissible
values of $\lambda$ getting narrower as $\tau$ approaches this critical value.
To illustrate this stabilizing effect, we plotted numerical solutions of \eqref{eq:N2react}
for $\tau=0.85$ and $\lambda\in\{0.04, 0.25, 0.45\}$ represented by the solid line in Fig. \ref{fig:4},
while, for comparison, the dashed line shows the solution without anticipation (i.e., $\lambda=0$).
We see that for $\lambda=0.02$ the effect of anticipation significantly reduces the amplitude
of oscillations of the solution, but still is too weak to ultimately stabilize it.
The choice $\lambda=0.25$ gives stability of the zero steady state,
but increasing $\lambda$ to $0.45$ makes it unstable again.
Unfortunately, we are not able to provide analytical insights into
this phenomenon. In particular, 	established methods for asymptotic analysis
of functional differential equations (Lyapunov-type functionals, characteristic polynomials)
deliver sufficient/necessary conditions for stability of the form of smallness
assumptions on the parameters; see \eqref{cond:react}.
However, in Figs. \ref{fig:3} and \ref{fig:4}
we observe stabilization for moderate levels of anticipation, where the value of $\lambda$
is neither too small nor too large. We therefore hypothesize that this effect is not
explainable by established analytical methods and new approaches need to be
developed. We leave it as subject of future work.

\begin{figure} \centering
 \includegraphics[width=.5\textwidth]{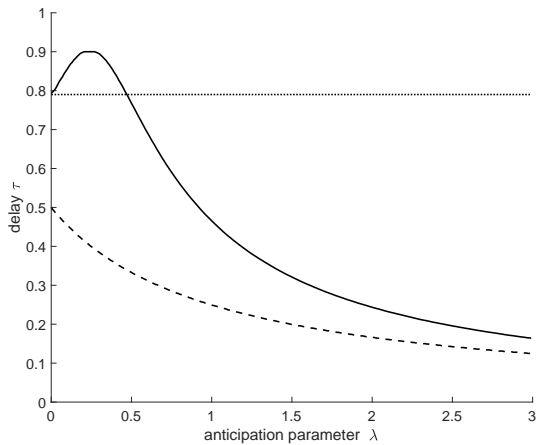}
 \caption{
Systematic simulations of \eqref{eq:N2react} subject to the constant initial datum $x^0(t) \equiv 1$ for $t\in [-\tau,0]$,
for the parameter range $\lambda\in [0,3]$ and $\tau\in [0,1]$.
The solid curve represents the threshold for stability
of the zero steady state obtained numerically -- the values of $(\lambda,\tau)$ below the solid curve
give $x(t)\to 0$ for large $t>0$, while the values above the curve produce diverging $x(t)$.
The dashed curve represents the analytical sufficient condition \eqref{cond:suffN2}.
The critical value $\tau=\pi/4$ for the system without anticipation is indicated by the dotted line.
 \label{fig:3}}
\end{figure}

\begin{figure} \centering
 \includegraphics[width=.32\textwidth]{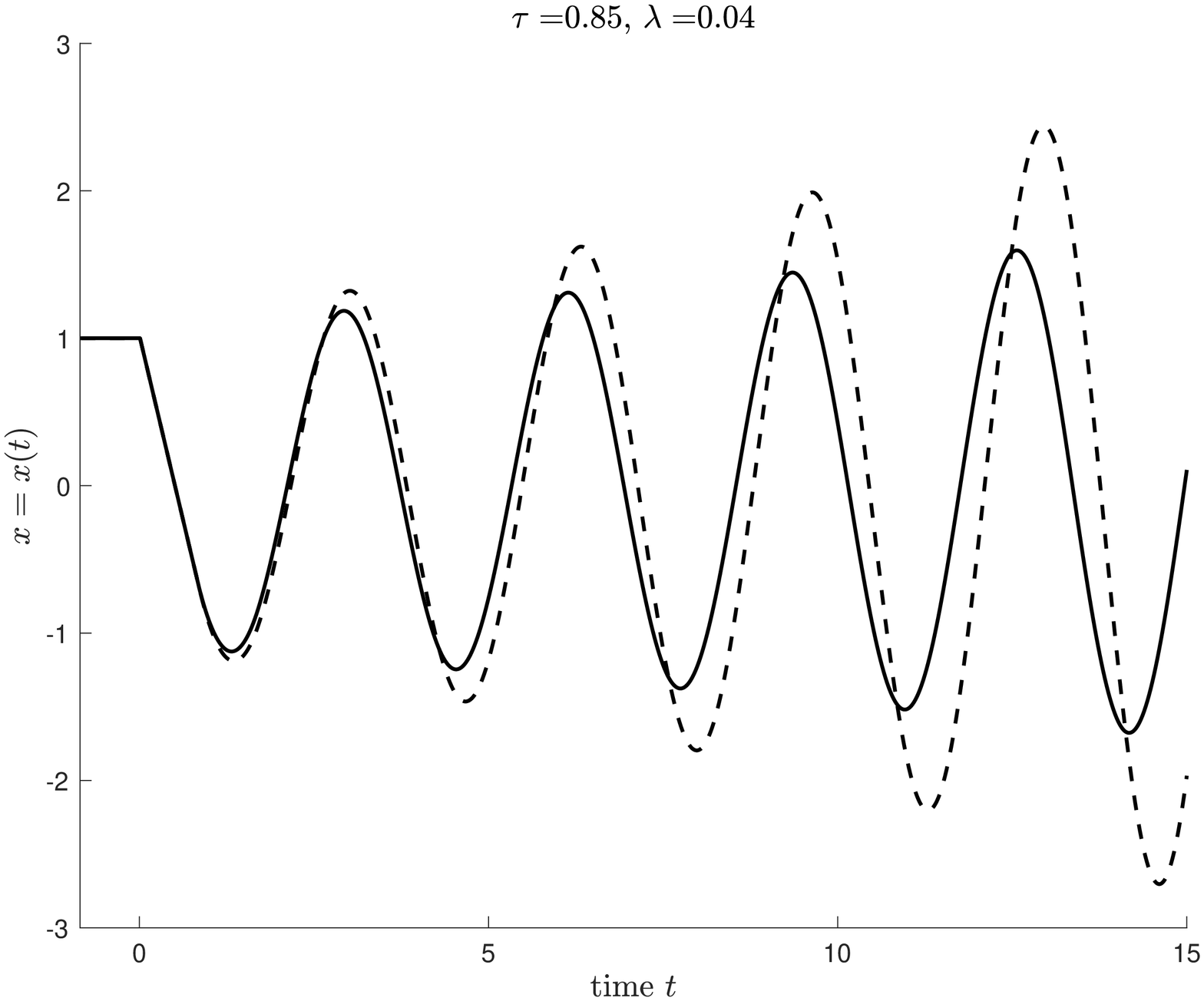}
 \includegraphics[width=.32\textwidth]{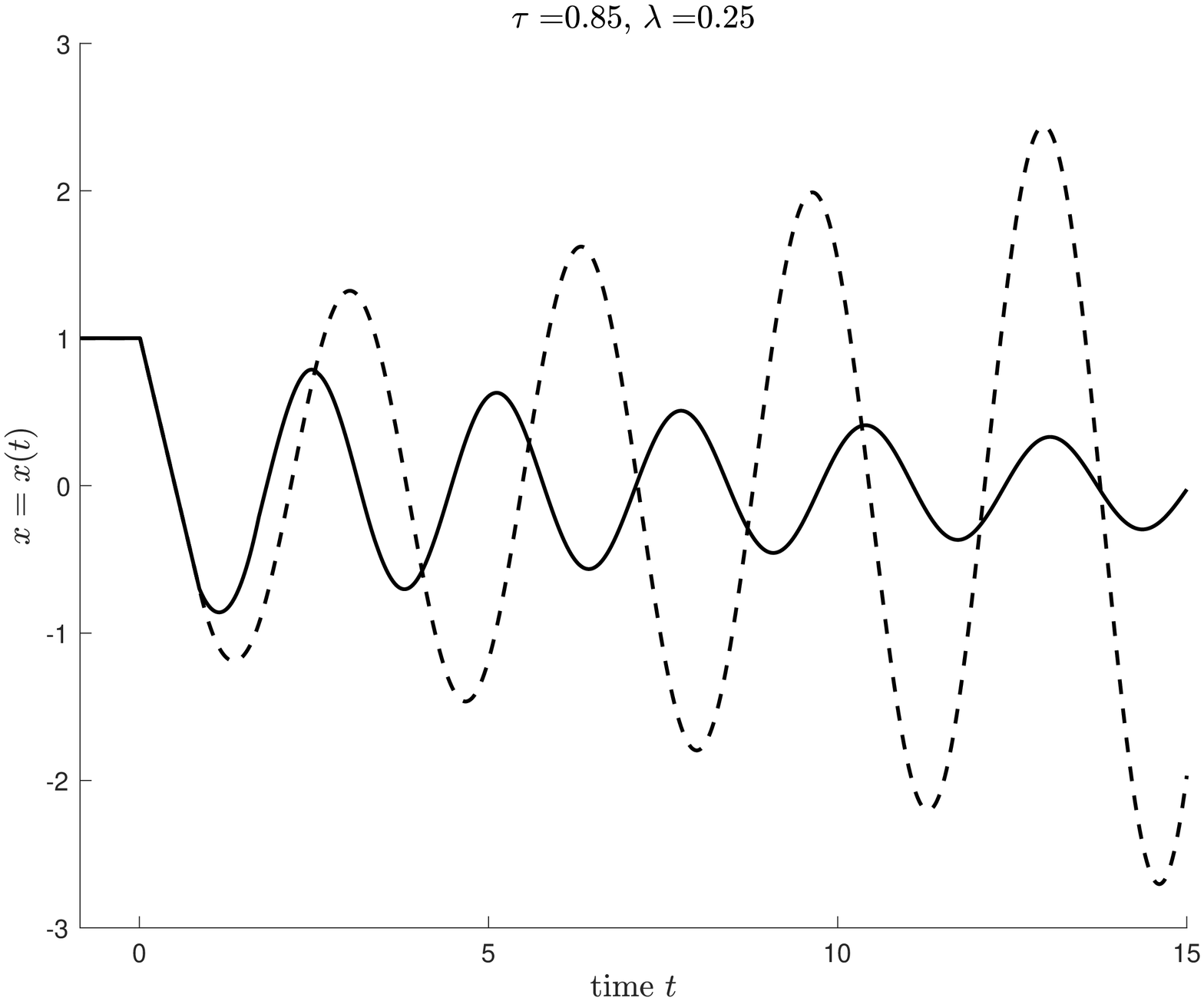}
 \includegraphics[width=.32\textwidth]{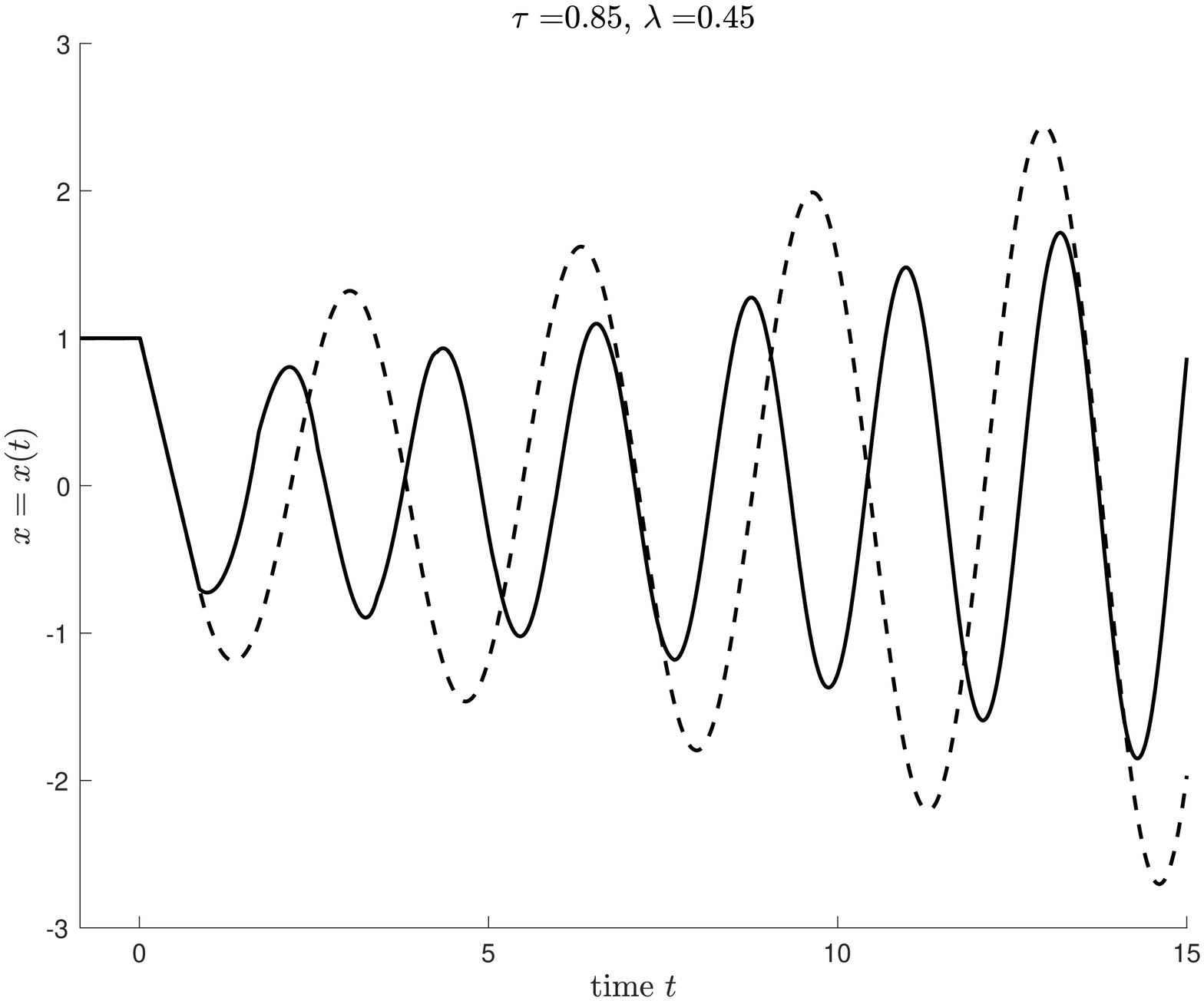}
 \caption{Numerical integration of \eqref{eq:N2react} subject to the constant initial datum $x^0 \equiv 0$,
 with $\tau=0.85$ and $\lambda=0.04$ (solid line in left panel), $\lambda=0.25$ (solid line in middle panel)
 and $\lambda=0.45$ (solied line in right panel).
 For comparison, numerical solution of \eqref{eq:N2react} with $\tau=0.85$ and $\lambda=0$ (no anticipation)
 is indicated by dashed line in all three panels.
 \label{fig:4}}
\end{figure}

\section{Asymptotic consensus with transmission delay - proof of Theorem \ref{thm:prop}}\label{sec:prop}
In this Section we provide a proof of Theorem \ref{thm:prop}.
It is based on the following geometrical lemma, presented in \cite{Has-SIADS},
based on the notion of coefficient of ergodicity introduced by Dobrushin \cite{Dobrushin}
in the context of Markov chains.

\begin{lemma}\label{lem:geom}
Let $N\geq 2$ and $\{x_1, \dots, x_N\}\subset \R^d$ be any set of vectors in $\R^d$.
Fix $i, k \in [N]$ such that $i\neq k$ and let
$\eta^i_j \geq 0$ for all $j\in[N] \setminus\{i\}$, 
and $\eta^k_j \geq 0$ for all $j\in[N] \setminus\{k\}$, 
such that
\[ 
   \sum_{j\neq i} \eta^i_j = 1,\qquad \sum_{j\neq k} \eta^k_j = 1.
\]
Denote
\[ 
   \mu := \min \left\{ \min_{j\neq i} \eta^i_j,\; \min_{j\neq k} \eta^k_j \right\} \geq 0.
\]
Then
\[   
   \left| \sum_{j\neq i} \eta^i_j x_j - \sum_{j\neq k} \eta^k_j x_j \right| \leq (1-(N-2)\mu) \max_{i,j \in [N]}|x_i - x_j|.
\]
\end{lemma}

We refer to \cite[Lemma 3.2]{Has-SIADS} for the proof. Note that, clearly,
Lemma \ref{lem:geom} only provides a nontrivial bound if $N\geq 3$;
however, the case $N=2$ was already discussed in Section \ref{subsec:N2prop}.

Let us recall that due to the positivity assumption \eqref{psi:pos},
all off-diagonal elements $\psi_{ij}$, $i\neq j$, are strictly positive.
Therefore, introducing the notation
\[  
    \upsi := \min_{i\neq j} \psi_{ij},
\]
we have $\upsi >0$. We further denote
\(   \label{def:gamma}
   \gamma:=1-(N-2)\upsi.
\)
Due to the stochasticity assumption \eqref{psi:conv} we have
\[
   1 = \sum_{j\neq i} \psi_{ij} \geq (N-1)\upsi,
\]
so that
\[
   \gamma \geq 1-\frac{N-2}{N-1} > 0.
\]
Let us also recall the assumption that there exists some $T_0>0$
and fixed indices $I,J \in [N]$ such that \eqref{IJ} holds for $t\geq T_0$.
We then define the functional
\(  \label{def:L}
   \mathcal{L}(t) := \frac12 \left[ \left( x_{I} - \lambda\tau \bPsi^{I}\right) - \left(x_{K} - \lambda\tau \bPsi^{K} \right) \right]^2
       + \left[ \frac{1+\lambda\tau}{2\gamma} - \lambda\tau \right] \gamma^2  \int_{t-\tau}^t d_x(s)^2 \d s,
\)
where we introduced the short-hand notation
\(  \label{def:Psii}
   \Psi^i(t) := \sum_{j\neq i} \psi_{ij} x_j(t),
\)
for $i\in [N]$. Note that we again apply the convention $\bPsii = \Psi^i(t-\tau)$.
Observe that since $\gamma\in (0,1]$, the term $\left[ \frac{1+\lambda\tau}{2\gamma} - \lambda\tau \right]$
is nonnegative whenever $\lambda\tau\leq 1$, and then $\mathcal{L}\geq 0$ by definition.
We now derive a decay estimate on $\mathcal{L}=\mathcal{L}(t)$.

\begin{lemma} \label{lem:diss}
Let the matrix of communication weights $(\psi_{ij})_{i,j\in [N]}$
verify the stochasticity \eqref{psi:conv} and positivity \eqref{psi:pos}
assumptions. Let $\lambda\tau\leq 1$ and $\gamma\in (0,1]$
be given by \eqref{def:gamma}.
Then for $\mathcal{L}=\mathcal{L}(t)$ given by \eqref{def:L} we have along solutions of \eqref{eq:prop},
for $t> T_0$,
\(   \label{L:diss}
   \tot{}{t} \mathcal{L}(t) \leq - (1-\gamma) (1-\lambda\tau\gamma) \, d_x(t)^2.
\)
\end{lemma}

\begin{proof}
We have for all $t>T_0$,
\[
    \frac12 \tot{}{t} \left[ \left( x_I - \lambda\tau \bPsi^{I}\right) - \left(x_K - \lambda\tau \bPsi^{K} \right) \right]^2  &=&
    - \Bigl( x_I - x_K - \lambda\tau \left( \bPsi^{I} - \bPsi^{K} \right) \Bigr) \cdot
      \Bigl( x_I - x_K - \left( \bPsi^{I} - \bPsi^{K} \right) \Bigr)  \\
      &=&
      - |x_I - x_K|^2 - \lambda\tau \left| \bPsi^{I} - \bPsi^{K} \right|^2 + (1+\lambda\tau) (x_I-x_K)\cdot \left( \bPsi^{I} - \bPsi^{K} \right).
\]
We apply the Cauchy-Schwarz inequality with $\delta>0$ for the last term,
\[
   (1+\lambda\tau) (x_I-x_K)\cdot \left( \bPsi^{I} - \bPsi^{K} \right) \leq
      \frac{(1+\lambda\tau) \delta}{2} |x_I - x_K|^2 + \frac{1+\lambda\tau}{2\delta} \left| \bPsi^{I} - \bPsi^{K} \right|^2,
\]
which gives, using \eqref{IJ},
\[
    \frac12 \tot{}{t} \left[ \left( x_I - \lambda\tau \bPsi^{I}\right) - \left(x_K - \lambda\tau \bPsi^{K} \right) \right]^2 \leq
      \left[ \frac{(1+\lambda\tau) \delta}{2} - 1\right] d_x(t)^2
       +  \left[ \frac{1+\lambda\tau}{2\delta} - \lambda\tau \right] \left| \bPsi^{I} - \bPsi^{K} \right|^2.
\]
With Lemma \ref{lem:geom}, recalling \eqref{def:gamma}, we have
\(   \label{geom_arg}
   \left| \bPsi^{I} - \bPsi^{K} \right| = \left| \sum_{j\neq i} \psi_{ij}\wtx_j - \sum_{j\neq k} \psi_{kj}\wtx_j \right|
     \leq (1 - (N-2) \upsi) \max_{i,j \in [N]}|\wtx_I - \wtx_j|
     = \gamma \, d_x(t-\tau).
\)
Consequently, making the assumption
\( \label{ass:delta}
   \delta \leq \frac{1+\lambda\tau}{2\lambda\tau},
\)
so that $\left[ \frac{1+\lambda\tau}{2\delta} - \lambda\tau \right] \geq 0$, we have
\[
   \frac12 \tot{}{t} \left[ \left( x_I - \lambda\tau \bPsi^{I}\right) - \left(x_K - \lambda\tau \bPsi^{K} \right) \right]^2 &\leq&
      \left[ \frac{(1+\lambda\tau) \delta}{2} - 1\right] d_x(t)^2 \\
     &&\qquad  +  \left[ \frac{1+\lambda\tau}{2\delta} - \lambda\tau \right] \gamma^2 \, d_x(t-\tau)^2.
\]
Therefore, defining
\[
   \mathcal{L}_\delta(t) := \frac12 \left[ \left( x_I - \lambda\tau \bPsi^{I}\right) - \left(x_K - \lambda\tau \bPsi^{K} \right) \right]^2
       + \left[ \frac{1+\lambda\tau}{2\delta} - \lambda\tau \right] \gamma^2  \int_{t-\tau}^t d_x(s)^2 \d s,
\]
we have
\[
   \tot{}{t} \mathcal{L}_\delta(t) \leq \left(
          \left[ \frac{(1+\lambda\tau) \delta}{2} -1 \right] + \left[ \frac{1+\lambda\tau}{2\delta} - \lambda\tau \right] \gamma^2  \right) d_x(t)^2.
\]
Minimization of the right-hand side in $\delta>0$ motivates the choice $\delta:=\gamma$, 
which identifies $\mathcal{L}_\delta$ with $\mathcal{L}$ defined in \eqref{def:L}, and
\[
   \tot{}{t} \mathcal{L}(t) \leq  
      - (1-\gamma) (1-\lambda\tau\gamma) \, d_x(t)^2.
\]
We finally note that with $\delta:=\gamma$ and $\lambda\tau\leq 1$, assumption \eqref{ass:delta}
is indeed verified since its right-hand side is larger or equal $1$ while $\delta=\gamma \leq 1$.
\end{proof}

Obviously, for $\lambda\tau\leq 1$ and $\gamma\in (0,1)$, Lemma \ref{lem:diss}
states that $\mathcal{L}$ is a monotonically decaying functional,
and the decay is strict unless $d_x\equiv 0$. However, this per se does not
imply that $d_x(t)\to 0$ as $t\to\infty$. Moreover, we cannot apply the
Lyapunov theorem here (strictly speaking, its version adapted
to neutral functional differential equations, see \cite[Theorem 8.1 of Section 9]{Hale-book}),
since system \eqref{eq:prop} has a continuum of steady states characterized by
$x_i \equiv x_j$ for all $i,j\in [N]$, and the asymptotic one is not determined
by the initial datum due to the lack of conserved quantities.
Consequently, to conclude the convergence to consensus, we shall
employ Barbalat's lemma \cite{Barbalat}. For its application we need
uniform continuity of $d_x=d_x(t)$, which is a direct consequence
on the following uniform a-priori estimate on the derivative of the trajectories
$x_i=x_i(t)$.


\begin{lemma}
Let $\lambda\tau\leq 1$ and denote
\(  \label{ass:M}
    M:= \max_{t\in [-\tau,0]} \max_{i\in [N]} \left\{ |x_i^0(t)|, |\dot x_i^0(t)| \right\}. 
\)
Then for all $t>0$ and $i\in [N]$,
\(   \label{claim:M}
     |x_i(t)|  \leq (1+\lambda\tau) M, \qquad 
     \left| \dot x_i(t) \right|  \leq 
        3 (1+\lambda\tau) M. 
\)
\end{lemma}

\begin{proof}
Let us denote, for $k\in\N$,
\[
   y_i^k(t) := x_i(t+(k-1)\tau) \qquad\mbox{for } t\in[0,\tau],\quad i\in [N].
\]
Then \eqref{eq:prop} is rewritten in terms of $y_i^k$ as
\[
   \dot y_i^{k+1} = \sum_{j=1}^N \psi_{ij} \left( y_j^k + \lambda\tau \dot y_j^k - y_i^{k+1} \right).
\]
A simple induction argument shows that for all $k\in\N$ and $i\in [N]$,
\(   
   \nonumber
   \dot y_{i}^{k+1} = 
     (1-\lambda\tau) \left[ \sum_{j_k=1}^N \psi_{i,j_k} y_{j_k}^k 
       +  \lambda\tau \sum_{j_k=1}^N \sum_{j_{k-1}=1}^N \psi_{i,j_k} \psi_{j_{k},j_{k-1}} y_{j_{k-1}}^{k-1} \right.
       + \ldots \\
       \ldots + \left.
       (\lambda\tau)^{k-1} \sum_{j_k=1}^N \ldots \sum_{j_1=1}^N \psi_{i,j_k} \ldots \psi_{j_2,j_1} y_{j_1}^1 \right]  \label{mad} \\
           +   (\lambda\tau)^{k} \sum_{j_k=1}^N \ldots \sum_{j_0=1}^N \psi_{i,j_k} \ldots \psi_{j_1,j_0} y_{j_0}^0  \nonumber\\       
           +   (\lambda\tau)^{k+1} \sum_{j_k=1}^N \ldots \sum_{j_0=1}^N \psi_{i,j_k} \ldots \psi_{j_1,j_0} \dot y_{j_0}^0  - y_{i}^{k+1}.
           \nonumber
\)
We shall first prove the first claim of \eqref{claim:M}, i.e., that for all $k\in\N$,
\(   \label{claim:M:proof}
     |y_i^k(t)|  \leq (1+\lambda\tau) M \qquad\mbox{for all } t\in [0,\tau],\quad i\in [N].
\)
Obviously, $y_i^0 \equiv x_i^0$ for all $i\in [N]$, so that \eqref{claim:M:proof} holds for $k=0$ due to \eqref{ass:M}.
We proceed inductively, assuming that \eqref{claim:M:proof} holds for $k=0,\dots,K$.
With the Cauchy-Schwarz inequality, using $\lambda\tau\leq 1$ and \eqref{psi:conv}, we infer from \eqref{mad},
\[
     \frac12 \tot{}{t} \left|y_{i}^{K+1} (t)\right|^2 &\leq&
           \left( (1-\lambda\tau)\sum_{k=0}^{K-1} (\lambda\tau)^k + (\lambda\tau)^{K} + (\lambda\tau)^{K+1} \right) M \left|y_{i}^{K+1} (t)\right| - \left|y_{i}^{K+1} (t)\right|^2
           \\
           &=& \Bigl[ \left( 1+ (\lambda\tau)^{K+1} \right) M - \left|y_{i}^{K+1} (t) \right| \Bigr] \left|y_{i}^{K+1} (t) \right|
\]
for $t\in (0,\tau)$.
Therefore, for almost all $t\in (0,\tau)$,
\(    \label{ind:2}
     \tot{}{t} \left|y_{i}^{K+1} (t)\right| \leq  \left( 1+ (\lambda\tau)^{K+1} \right) M - \left|y_{i}^{K+1} (t) \right|.
\)
Then, integration in time gives
\[
   \left|y_{i}^{K+1} (t)\right| &\leq& e^{-t} \left|y_{i}^{K+1} (0)\right| + \left(1-e^{-t} \right) \left( 1+ (\lambda\tau)^{K+1} \right) M \\
      &\leq&  e^{-t} (1+\lambda\tau) M + \left(1-e^{-t} \right) \left( 1+ (\lambda\tau)^{K+1} \right) M  \\
      &\leq&  (1+\lambda\tau) M,
 \]
for $t\in[0,\tau]$,
where we used the induction hypothesis $\left|y_{i}^{K+1} (0)\right| = \left|y_{i}^{K} (\tau)\right| \leq (1+\lambda\tau) M$ for the second inequality
and $\lambda\tau\leq 1$ for the third one.

Finally, using \eqref{claim:M:proof} in \eqref{mad} gives
\[
   \left| \dot y_{i}^{K+1} \right| &\leq& \left( (1-\lambda\tau)\sum_{k=0}^{K-1} (\lambda\tau)^k + (\lambda\tau)^{K} + (\lambda\tau)^{K+1}  + 1 \right) (1+\lambda\tau) M \\
      &=& \left( 2 + (\lambda\tau)^{K+1} \right) (1+\lambda\tau) M
\]
for all $K\in\N$, which with $\lambda\tau\leq 1$ implies the second claim of \eqref{claim:M}.
\end{proof}

We now provide the proof of Theorem \ref{thm:prop}.

\begin{proof}
An integration of \eqref{L:diss} on the interval $(T_0,t)$ gives
\[
   (1-\gamma) (1-\lambda\tau\gamma) \int_{T_0}^t d_x(s)^2 \d s \leq \mathcal{L}(T_0) - \mathcal{L}(t) \leq \mathcal{L}(T_0).
\]
Therefore, since for $N\geq 3$ we have $\gamma<1$ by \eqref{def:gamma},
the squared diameter $d^2_x=d^2_x(t)$ must be integrable
as $t\to\infty$. Moreover, with \eqref{claim:M} we have
\[
   \left| \tot{}{t} d_x(t)^2 \right| \leq 2 |x_i-x_k| |\dot x_i - \dot x_k| \leq 24 (1+\lambda\tau)^2 M^2.
\]
Then, an application of the Barbalat lemma \cite{Barbalat} implies that $\lim_{t\to\infty} d_x(t) = 0$.
\end{proof}

\begin{remark} \label{rem:IJ}
Let us now explain why we had to adopt the assumption \eqref{IJ} of Theorem \ref{thm:prop}.
Observe that, due to the continuity of the trajectories $x_i=x_i(t)$,
there is an at most countable system of open, mutually disjoint
intervals $\{\mathcal{I}_\sigma\}_{\sigma\in\N}$ such that
$$
   \bigcup_{\sigma\in\N} \overline{\mathcal{I}_\sigma} = [0,\infty)
$$
and for each ${\sigma\in\N}$ there exist indices $i(\sigma)$, $k(\sigma)$
such that
\(   \label{dxik}
   d_x(t) = |x_{i(\sigma)}(t) - x_{k(\sigma)}(t)| \quad\mbox{for } t\in \mathcal{I}_\sigma.
\)
Obviously, assumption \eqref{IJ} means that the system $\{\mathcal{I}_\sigma\}_{\sigma}$
can be chosen of finite cardinality, with the last interval being $(T_0,\infty)$.

If assumption \eqref{IJ} was not met, i.e., the system $\{\mathcal{I}_\sigma\}_{\sigma\in\N}$ was countably infinite,
we would have to replace the functional \eqref{def:L} by
\[
   \mathcal{L}'(t) := \frac12 \left[ \left( x_{i(\sigma)} - \lambda\tau \bPsi^{i(\sigma)}\right) - \left(x_{k(\sigma)} - \lambda\tau \bPsi^{k(\sigma)} \right) \right]^2
       + \left[ \frac{1+\lambda\tau}{2\gamma} - \lambda\tau \right] \gamma^2  \int_{t-\tau}^t d_x(s)^2 \d s,
\]
with $i=i(\sigma)$, $k=k(\sigma)$ given by \eqref{dxik}.
Then, a minimal adaptation of the proof of Lemma \ref{lem:diss}
would give the dissipation estimate \eqref{L:diss} for $\mathcal{L}'=\mathcal{L}'(t)$,
for all $t\in \bigcup_{\sigma\in\N} \mathcal{I}_\sigma$, i.e., for almost all $t>0$.
However, the functional $\mathcal{L}'(t)$ is not globally continuous
since the terms $\bPsi^{i(\sigma)}$, $\bPsi^{k(\sigma)}$
may have jumps across the boundaries of the intervals $\mathcal{I}_\sigma$.
Although the jumps can be controlled in spirit of \eqref{geom_arg},
this control is not sufficient to finally conclude that $\lim_{t\to\infty} d_x(t) = 0$.
This is the reason why we had to admit the assumption \eqref{IJ}
in the formulation of Theorem \ref{thm:prop}.
\end{remark}

\section{Asymptotic consensus with reaction delay - proof of Theorem \ref{thm:react}}\label{sec:react}

Let us recall that Theorem \ref{thm:react} assumes the matrix of communication weights $(\psi_{ij})_{i,j\in [N]}$ to be
symmetric
and irreducible. Since the values of the diagonal entries $\psi_{ii}$ are irrelevant for
the dynamics of \eqref{eq:react}, we can formally set them such that all row and column sums are the same.
In particular, without loss of generality (up to an eventual rescaling of time), we shall assume that
$(\psi_{ij})_{i,j\in [N]}$ is a bi-stochastic matrix, i.e.,
\(  \label{psi:bistoch}
    \sum_{j=1}^N \psi_{ij} = \sum_{j=1}^N \psi_{ji} = 1 \qquad\mbox{for all } i\in [N],
\)
The symmetry
implies that the mean opinion $X=X(t)$ defined in \eqref{def:mean}
is conserved along the solutions of \eqref{eq:react}. Therefore, without loss of generality,
we shall assume that $X(t) \equiv 0$ for all $t\geq 0$.

For $t\geq 0$ we define the quantity
\(  \label{def:D}
   D(t):= \sum_{i=1}^N \sum_{j=1}^N {\psi}_{ij} |{x}_{j} - {x}_{i}|^{2},
\)
with the convention $\widetilde{D} := D(t-\tau)= \sum_{i=1}^N \sum_{j=1}^N {\psi}_{ij} |\widetilde{x}_{j} - \widetilde{x}_{i}|^{2}$.
We also introduce the short-hand notation
\(  \label{def:Phii}
   \Phi^i := \sum_{j=1}^N \psi_{ij} (x_j - x_i),
\)
and, again, $\bPhii = \Phi^i(t-\tau)$.

\begin{lemma} \label{lem:sd1}
For any $\eps>0$ and $\delta>0$ we have,
along the solutions of \eqref{eq:react} with bi-stochastic interaction weights \eqref{psi:bistoch},
\(  \label{est:sd1}
   \sum_{i=1}^N \bPhii \cdot x_i &\leq&
         \frac{\delta-1}{2} D(t-\tau)
           + \frac{(1+\eps)\tau}{2\delta} \int_{t-\tau}^{t} {D}(s-\tau)\,\d s \\
           &&\qquad
           + \frac{\lambda^2\tau^2}{\delta} \left(1+\eps^{-1}\right) \sum_{i=1}^N \left( \left| \Phi^i(t-\tau) \right|^2 + \left| \Phi^i(t-2\tau) \right|^2 \right),
           \nonumber
\)
for all $t>\tau$.
\end{lemma}

\begin{proof}
With \eqref{eq:react} we have
\[
   \sum_{i=1}^N \bPhii \cdot x_i
     &=& \sum_{i=1}^N \sum_{j=1}^N \psi_{ij} (\widetilde x_j - \widetilde x_i) \cdot x_i \\
     &=& \sum_{i=1}^N \sum_{j=1}^N \psi_{ij} (\widetilde x_j - \widetilde x_i) \cdot \widetilde x_i
      + \sum_{i=1}^N \sum_{j=1}^N \psi_{ij} (\widetilde x_j - \widetilde x_i) \cdot (x_i - \widetilde x_i).
\]
For the first term of the right-hand side we apply the standard
symmetrization trick (exchange of summation indices $i\leftrightarrow j$, using the symmetry of ${\psi}_{ij} = {\psi}_{ji}$),
\[
      \sum_{i=1}^N \sum_{j=1}^N \psi_{ij} (\widetilde x_j - \widetilde x_i) \cdot \widetilde x_i
      =  - \frac{1}{2} \sum_{i=1}^N \sum_{j=1}^N \psi_{ij} |\widetilde x_j - \widetilde x_i|^2
      = - \frac{1}{2} D(t-\tau) \,.
\]
For the second term we use the Young inequality with $\delta>0$,
\[
   \left|
      \sum_{i=1}^N \sum_{j=1}^N \psi_{ij} (\widetilde x_j - \widetilde x_i) \cdot (x_i - \widetilde x_i) \right|
     &\leq&
       \frac{\delta}{2} \sum_{i=1}^N \sum_{j=1}^N \psi_{ij} |\widetilde x_j - \widetilde x_i|^2
         + \frac{1}{2\delta} \sum_{i=1}^N \sum_{j=1}^N \psi_{ij} |x_i - \widetilde x_i|^2  \\
     &=&    
          \frac{\delta}{2} D(t-\tau)
         + \frac{1}{2\delta} \sum_{i=1}^N |x_i - \widetilde x_i|^2,
\]
where we used \eqref{psi:bistoch} in the second line.
Hence,
\(  \label{lem:sd1:1}
   \sum_{i=1}^N \bPhii \cdot x_i \leq
          \frac{\delta-1}{2} D(t-\tau) + \frac{1}{2\delta} \sum_{i=1}^N |x_i - \widetilde x_i|^2.
\)
Next, for $t>\tau$ we use \eqref{eq:react} to evaluate the difference
$x_{i}-\widetilde{x}_{i}$,
\[
   x_{i}-\widetilde{x}_{i} &=& \int_{t-\tau}^{t} \dot x_{i}(s)\, \d s \\
    &=&  \int_{t-\tau}^{t} \Phi^i(s-\tau) + \lambda\tau \dot\Phi^i(s-\tau)  \, \d s \\
    &=&  \int_{t-\tau}^{t} \Phi^i(s-\tau) \, \d s + \lambda\tau \left( \Phi^i(t-\tau) - \Phi^i(t-2\tau) \right).
\]
Taking the square, using the inequality $(a+b)^2 \leq (1+\eps)a^2 + (1+\eps^{-1})b^2$ with $\eps>0$,
and summing over $i\in [N]$ gives
\[
   \frac{1}{2\delta} \sum_{i=1}^N |x_i - \widetilde{x}_{i}|^{2} &\leq&
       \frac{1+\eps}{2\delta} \sum_{i=1}^N \left|  \int_{t-\tau}^{t} \Phi^i(s-\tau)\, \d s  \right|^2
         + \frac{\lambda^2\tau^2}{2\delta} \left(1+\eps^{-1}\right) \sum_{i=1}^N \left| \Phi^i(t-\tau) - \Phi^i(t-2\tau)  \right|^2 \\
       &\leq&
       \frac{(1+\eps)\tau}{2\delta} \sum_{i=1}^N  \int_{t-\tau}^{t} \left| \Phi^i(s-\tau)  \right|^2 \, \d s
         + \frac{\lambda^2\tau^2}{\delta} \left(1+\eps^{-1}\right) \sum_{i=1}^N \left( \left| \Phi^i(t-\tau) \right|^2 + \left| \Phi^i(t-2\tau)  \right|^2 \right),
\]
where we used the Cauchy-Schwarz inequality in the second line.
The first term of the right-hand side is estimated using the discrete Jensen inequality, recalling the stochasticity \eqref{psi:bistoch},
\[
     \frac{(1+\eps)\tau}{2\delta} \sum_{i=1}^N \int_{t-\tau}^{t} | \Phi^i(s-\tau) |^2 \, \d s
   &=&
     \frac{(1+\eps)\tau}{2\delta} \sum_{i=1}^N \int_{t-\tau}^{t} \left| \sum_{j=1}^N {\psi}_{ij} ({x}_{j}(s-\tau) - {x}_{i}(s-\tau)) \right|^2 \, \d s \\   
   &\leq&
     \frac{(1+\eps)\tau}{2\delta} \sum_{i=1}^N \int_{t-\tau}^{t} \sum_{j=1}^N {\psi}_{ij} |{x}_{j}(s-\tau) - {x}_{i}(s-\tau)|^2 \, \d s \\
   &=&
     \frac{(1+\eps)\tau}{2\delta} \int_{t-\tau}^{t} {D}(s-\tau)\,\d s.
\]
Consequently, we arrive at
\[
   \frac{1}{2\delta} \sum_{i=1}^N |x_i - \widetilde{x}_{i}|^{2} \leq 
      \frac{(1+\eps)\tau}{2\delta} \int_{t-\tau}^{t} {D}(s-\tau)\,\d s
         + \frac{\lambda^2\tau^2}{\delta} \left(1+\eps^{-1}\right) \sum_{i=1}^N \left( \left| \Phi^i(t-\tau) \right|^2 + \left| \Phi^i(t-2\tau)  \right|^2 \right),
\]
and inserting into \eqref{lem:sd1:1} gives \eqref{est:sd1}.
\end{proof}

For $\eps>0$ and $\delta>0$ to be specified later, we define the functional
\(  \label{Lyap:sd}
   \mathcal{L}_{\eps,\delta} (t) := \frac12 \sum_{i=1}^N \left( x_i - \lambda\tau \bPhii \right)^2
      + \frac{\lambda^2\tau^2}{\delta} \left( 1 + \eps^{-1} \right) \int_{t-\tau}^t \sum_{i=1}^N \left| \Phi^i (s-\tau) \right|^2 \d s \\
      +\, \frac{(1+\eps)\tau}{2\delta} \int_{t-\tau}^t \int_\theta^t D(s-\tau) \d s\d\theta.
      \nonumber
\)

\begin{lemma}\label{lem:Led}
For $\eps:=2\lambda$ and $\delta:= (1 + 2\lambda)\tau$, we have
along the solutions of \eqref{eq:react} with bi-stochastic interaction weights \eqref{psi:bistoch},
\(  \label{Led:diss}
   \tot{}{t} \mathcal{L}_{\eps,\delta} (t) \leq  \left( (1+\lambda)\tau - \frac12 \right) D(t-\tau),
\)
for all $t> \tau$, with $\mathcal{L}_{\eps,\delta}=\mathcal{L}_{\eps,\delta}(t)$ defined in \eqref{Lyap:sd}.
\end{lemma}

\begin{proof}
Using \eqref{eq:react}, we have for all $i\in [N]$,
\[
   \frac12 \tot{}{t} \left( x_i - \lambda\tau \bPhii \right)^2 =
      \left( x_i - \lambda\tau \bPhii \right) \cdot \bPhii
       = - \lambda\tau \left| \bPhii \right|^2 + \bPhii\cdot x_i.
\]
Summing over $i\in [N]$ and using \eqref{est:sd1}, recalling the notation $\bPhii = \Phi^i(t-\tau)$,
\[
   \frac12 \tot{}{t} \sum_{i=1}^N \left( x_i - \lambda\tau \bPhii \right)^2 &\leq&
            \frac{\delta-1}{2} D(t-\tau)
           + \frac{(1+\eps)\tau}{2\delta} \int_{t-\tau}^{t} {D}(s-\tau)\,\d s \\
           && \qquad
           + \left( \frac{\lambda^2\tau^2}{\delta}\left(1+\eps^{-1}\right) - \lambda\tau \right) \sum_{i=1}^N \left| \Phi^i(t-\tau) \right|^2 \\
           && \qquad\qquad
              + \frac{\lambda^2\tau^2}{\delta} \left(1+\eps^{-1}\right) \sum_{i=1}^N \left| \Phi^i(t-2\tau)  \right|^2.
\]
Noting that
\[
   \frac{\lambda^2\tau^2}{\delta} \left(1+\eps^{-1}\right)  \tot{}{t} \int_{t-\tau}^t \sum_{i=1}^N \left| \Phi^i (s-\tau) \right|^2 \d s
      = \frac{\lambda^2\tau^2}{\delta}  \left(1+\eps^{-1}\right) \sum_{i=1}^N \left( \left| \Phi^i (t-\tau) \right|^2 - \left| \Phi^i (t-2\tau) \right|^2 \right),
\]
and
\[
    \frac{(1+\eps)\tau}{2\delta}  \tot{}{t} \int_{t-\tau}^t \int_\theta^t D(s-\tau) \d s\d\theta =
      \frac{(1+\eps)\tau^2}{2\delta} D(t-\tau) - \frac{(1+\eps)\tau}{2\delta} \int_{t-\tau}^t D(s-\tau) \d s,
\]
we have
\[
   \tot{}{t} \mathcal{L}_{\eps,\delta}(t) \leq \left( \frac{\delta-1}{2} + \frac{(1+\eps)\tau^2}{2\delta} \right) D(t-\tau)
      + \left( \frac{2\lambda^2\tau^2}{\delta} \left(1+\eps^{-1}\right) - \lambda\tau \right) \sum_{i=1}^N \left| \Phi^i(t-\tau) \right|^2.
\]
Next, the discrete Jensen inequality, recalling the stochasticity \eqref{psi:bistoch}, gives
\[
   \sum_{i=1}^N \left| \Phi^i(t-\tau) \right|^2 = \sum_{i=1}^N \left| \sum_{j=1}^N \psi_{ij} (\wtx_j-\wtx_i) \right|^2
     \leq \sum_{i=1}^N \sum_{j=1}^N \psi_{ij} |\wtx_j-\wtx_i|^2 = D(t-\tau).
\]
Consequently, assuming that $\lambda$, $\tau$ and $\eps$, $\delta$ are such that
\( \label{cond:lted}
    \frac{2\lambda^2\tau^2}{\delta} \left(1+\eps^{-1}\right) - \lambda\tau \geq 0,
\)
we have
\[
   \tot{}{t} \mathcal{L}_{\eps,\delta} (t) \leq
       \left( \frac{\delta-1}{2} + \frac{(1+\eps)\tau^2}{2\delta} + \frac{2\lambda^2\tau^2}{\delta} \left(1+\eps^{-1}\right) - \lambda\tau \right) D(t-\tau).
\]
We first minimize the right-hand side in $\eps>0$, which leads to $\eps:=2\lambda$ and
\[
   \tot{}{t} \mathcal{L}_{\eps,\delta} (t) \leq
       \left( \frac{\delta-1}{2} + \frac{(1+2\lambda)^2\tau^2}{2\delta} - \lambda\tau \right) D(t-\tau).
\]
Now we minimize the right-hand side in $\delta>0$. This gives the optimal choice $\delta:=(1+2\lambda)\tau$ and
\[
   \tot{}{t} \mathcal{L}_{\eps,\delta} (t) \leq  \left( (1+\lambda)\tau - \frac12 \right) D(t-\tau),
\]
which is \eqref{Led:diss}. It remains to check that 
condition \eqref{cond:lted} is verified.
Remarkably, a simple calculation shows that with $\eps:=2\lambda$ and $\delta:=(1+2\lambda)\tau$,
\eqref{cond:lted} holds with equality.
\end{proof}

We are now in position to present the proof of Theorem \ref{thm:react}.
\begin{proof}
We first note that, assuming $X=0$, the only stationary solution of \eqref{eq:react}
is the trivial solution $x_i=0$ for all $i\in [N]$.
Indeed, the stationary version of \eqref{eq:react} can be written as
\(  \label{Lx}
   -L \mathbf{x} = 0,
\)
where $\mathbf{x}=(x_1,\dots,x_N)$ and $L\in\R^{N\times N}$ is the discrete Laplacian
\[
   L_{ij} := -\psi_{ij} \quad\mbox{for}\quad i\neq j,\qquad
   L_{ii} := \sum_{j\neq i} \psi_{ij} \quad\mbox{for}\quad i\in [N].
\]
Since $(\psi_{ij})_{i,j\in [N]}$ is assumed to be irreducible,
basic theory \cite{Mohar} states that
the null space of $L$ is the diagonal
$\left\{ (x,\ldots,x)\in\R^{Nd}; x\in\R^d \right\}$.
But then, the assumption $X=0$ implies that
\eqref{Lx} holds if and only if $\mathbf{x}=0$.

Consequently, global asymptotic consensus for \eqref{eq:react}
is equivalent to the asymptotic stability of the trivial steady state $x_i=0$ for all $i\in [N]$.
But the asymptotic stability is implied by the fact that, if \eqref{cond:react} holds,
then, by Lemma \ref{lem:Led}, $\mathcal{L}_{\eps,\delta}$ with $\eps:=2\lambda$ and $\delta:= (1 + 2\lambda)\tau$
is a Lyapunov functional for \eqref{eq:react}.
Indeed, due to the irreducibility of $(\psi_{ij})_{i,j\in [N]}$,
\[
   D(t) = \sum_{i=1}^N \sum_{j=1}^N {\psi}_{ij} |{x}_{j} - {x}_{i}|^{2} = L\mathbf{x}\cdot\mathbf{x}
\]
vanishes if and only if $\mathbf{x}=0$.
Consequently, the right-hand side of \eqref{Led:diss}
is nonpositive and zero only for the trivial steady state.
An application of the Lyapunov theorem for systems of neutral functional differential equations,
see, e.g., \cite[Theorem 8.1 of Section 9]{Hale-book}, provides then the sought global
asymptotic stability of the trivial steady state $x_i=0$ for all $i\in [N]$.
\end{proof}

\section*{Acknowledgment}
JH acknowledges the support of the KAUST baseline funds.



\end{document}